\documentclass[12pt]{amsart}
\usepackage{amsmath, amsfonts, latexsym, array, amssymb}
\usepackage{psfrag}
\usepackage[all]{xy}
\usepackage{graphicx}
\usepackage{color}

\newtheorem{mythm}{Theorem}
\newtheorem{mycor}{Corollary}
\newtheorem{myfact}{Fact}

\newtheorem{thm}{Theorem}[section]
\newtheorem{cor}[thm]{Corollary}
\newtheorem{lem}[thm]{Lemma}

\newtheorem{prop}[thm]{Proposition}
\newtheorem{cla}[thm]{Claim}
\newtheorem{defn}[thm]{Definition}
\newtheorem{rem}[thm]{Remark}

\newcommand{\deleted}[1]{}

\newcommand{\comment}[1]{}

\newcommand\Diff{{\operatorname{Diff}}}
\newcommand\eps{\epsilon}

\newcommand\Lip{{\operatorname{Lip}}}

\newcommand\NN{{\mathbb N}}
\newcommand\Prob{{\operatorname{Prob}}}
\newcommand\Pq{{\mathcal P'}}
\newcommand\RR{{\mathbb R}}

\renewcommand\top{{\operatorname{top}}}
\newcommand\TT{{\mathbb T}}
\newcommand\ZZ{{\mathbb Z}}

\newcommand\mat[1]{{\left(\begin{matrix} #1 \end{matrix}\right)}}
\newcommand\alter[1]{{\left\{\begin{array}{ll} #1 \end{array}\right.}}

\begin{document}

\title[Diffeomorphisms without maximal measures]{$C^r$ surface diffeomorphisms with\\ no maximal entropy measure}
\author{J\'er\^ome Buzzi}
\address{C.N.R.S. \& D\'epartement de Math\'ematiques, Universit\'e Paris-Sud, 91405 Orsay, France}
\email{jerome.buzzi@math.u-psud.fr}
\thanks{}

\date{June 2011, May 2012}
\commby{}

\begin{abstract}
For any $1\leq r<\infty$, we build on the disk and therefore on any manifold, a $C^r$-diffeomorphism with no measure of maximal entropy.\\
$  $\\
{\sc R\'esum\'e.} Pour tout $1\leq r<\infty$, nous construisons, sur le disque et donc sur toute vari\'et\'e, un diff\'eomorphisme de classe $C^r$ sans mesure d'entropie maximale.
 \end{abstract}

\maketitle

\section{Introduction}

\subsection{Measures of maximal entropy}
Consider a diffeomorphism $f$ of a compact manifold $M$ with nonzero topological entropy $h_\top(f)>0$,  i.e., having an exponentially large number of orbits of a given length distinguishable with a fixed precision (see Section~\ref{subsec:def} for definitions). A natural question is whether there exist invariant measures $\mu$ "describing almost all orbits"  in the sense that their entropy $h(f,\mu)$ is maximal and equal to $h_\top(f)$. Such measures, if they exist, include ergodic ones which we will call \emph{maximal (entropy) measures}.  

The variational principle states that continuity is enough to ensure the existence of measures with entropy arbitrarily close to $h_\top(f)$.   Sufficient conditions for achieving this value include uniform hyperbolicity and, by a classical result of Newhouse  \cite{newhouse1989}, $C^\infty$ smoothness (or a combination \cite{buzziruette2006}). 
Newhouse theorem is sharp: for every $r<\infty$, there exist $C^r$ smooth dynamical systems on compact manifolds with no maximal measures.

The first counter-examples are due to Misiurewicz \cite{misiurewicz1973} who built $C^r$ diffeomorphisms of the $4$-torus with no measure of maximal entropy for all $r<\infty$. This author has given simple examples for maps on the interval \cite{Buzzi1997} (see also \cite{ruette2002}) and for continuous, piecewise affine maps of the square \cite{buzziNoMaxAff2009}. However the case of $C^r$ diffeomorphisms in dimensions $2$ and $3$ has remained open. The present paper relies on a construction of Newhouse to provide such counter-examples.

\subsection{Definitions}\label{subsec:def}
For a real number $r\geq0$, one says that a map $f:M\to N$ is $C^r$ smooth if it is $[r]$ times\footnote{$[r]:=\max\{s\in\ZZ:s\leq r\}$.} differentiable with $D^{[r]}f$ H\"older continuous of exponent $r-[r]$ (or just continuous if $r=[r]$, i.e., if $r\in\NN$). Let $D:=\{(x,y)\in\RR^2:x^2+y^2\leq 4\}$. If $M$ is a manifold with boundary, $\Diff^r_0(M)$ denotes the set of $C^r$ diffeomorphisms of $M$ which coincide with the identity on a neighborhood of the boundary of $M$. $\Diff^r_0(M)$ is endowed with the $C^r$ topology defined by the $C^r$ norm on any finite $C^r$ atlas of $M$.

We recall\footnote{We refer to \cite{walters2000introduction} or \cite{katok1997introduction} for background on the entropy theory of continous maps of compact metric spaces.} that the \emph{topological entropy} of a continuous map $f:M\to M$ of a compact metric space is: $h_\top(f):=\lim_{\eps\to0} h_\top(f,\eps)$ with $h_\top(f,\eps):=\limsup_{n\to\infty} \frac1n\log r_f(\eps,n)$ and
 $$
     r_f(\eps,n):=\max\{\#S:S\subset M\text{ and }\forall x,y\in S\; x=y \text{ or } d_{f,n}(x,y)\geq\eps\}
 $$
using the Bowen-Dinaburg's distances $d_{f,n}(x,y):=\max\{ d(f^kx,f^ky):0\leq k<n\}$ for $n\geq1$. One can say loosely that the topological entropy counts the number of orbits.

There is a \emph{measure-theoretic entropy}:  to each measure\footnote{From now on, measure means invariant probability Borel measure.} $\mu$ is associated a number $h(f,\mu)\in[0,\infty]$, which "counts the number of orbits seen" by the measure $\mu$.
The \emph{variational principle} relates the two notions of entropy:
 \begin{equation}\label{eq:varprin}
    h_\top(f) = \sup_{\mu\in\Prob(f)} h(f,\mu).
 \end{equation}
Maximal entropy measures are exactly those achieving this supremum (if they exist) and can be interpreted as those seeing all the orbits (on a logarithmic scale).

Finally, any measure $\mu$ has an ergodic decomposition: $\mu=\int \mu_c \, \hat\mu(dc)$ with $\mu_c$ an ergodic measure for $\hat\mu$-a.e. $c$, and the entropy function is \emph{harmonic} over the set $\Prob(f)$  of invariant probability measures with the vague topology: 	 $h(f,\mu)=\int h(f,\mu_c)\, \hat\mu(dc)$. Thus, (i) to exclude the existence of invariant probability measures with maximal entropy, it is enough to exclude that of ergodic  ones; (ii) the supremum in the variational principle \eqref{eq:varprin} can be restricted to ergodic measures.

We will need the following two bounds on the entropy of a measure. First, say that a measure is \emph{quasi-periodic} if it is supported by a periodic orbit or inside a topological circle which is conjugate to a rotation. The entropy of a quasi-periodic measure is zero. Second, if the map is $C^1$, the \emph{Ruelle inequality} states (in the special case of a surface diffeomorphism) that for an ergodic invariant probability measure $\mu$:
 \begin{equation}\label{eq:Ruelle}
     h(f,\mu) \leq \lambda^u(f,\mu)^+ \text{ with } \lambda^u(f,\mu)\text{ the $\mu$-a.e. value of }\max_{v\in T^1_xM} \lambda(x,v)
 \end{equation}
where $s^+:=\max(s,0)$, $\lambda(x,v):=\liminf_{n\to\infty} \frac1n\log\|(f^n)'(x).v\|$ and $T^1_xM:=\{v\in T_xM:\|v\|_x=1\}$. $\lambda(x,v)$ is the Lyapunov exponent of $v\in T_xM$ whereas $\lambda^u(f,\mu)$ is the maximal Lyapunov exponent of $\mu$. Oseledets theorem (see, e.g., \cite{katok1997introduction}) ensures that the existence of such a $\mu$-a.e. constant value.

Let $\Lip(f):=\sup_{x\ne y} d(fx,fy)/d(x,y)$ be the Lipschitz constant of $f$ and set $\lambda(f):=\lim_{n\to\infty}\frac1n\log\Lip(f^n)$. We have $h_\top(f)\leq\lambda(f)$.

\subsection{Main results}
Let $\Pq(f)$ be the set of ergodic, invariant probability measures for $f$ which are not quasi-periodic. Let $\lambda^u(\Pq(f)) := \sup_{\nu\in\Pq(f)} \lambda^u(f,\nu)$. 

\begin{mythm}\label{mythm1}
There exists $f_0\in\Diff^\infty_0(D)$ with $h_\top(f_0)=0$ and the following properties.

Let  $1\leq r<\infty$ be any real number and let  $\mathcal U_0$ be any neighborhood of $f_0$ in $\Diff^r_0(M)$.
There exists $f\in\mathcal U_0$ such that:
 \begin{enumerate}
  \item $h_\top(f)=\frac1r\lambda(f_0)>0$;
  \item  
 $
     \forall \mu\in\Pq(f) \quad \lambda^u(f,\mu) < \lambda^u(\Pq(f)) = \frac1r\lambda(f_0).
 $
 \end{enumerate}
\end{mythm}

Using Ruelle's inequality \eqref{eq:Ruelle} and the zero entropy of quasi-periodic measures, we obtain:

\begin{mycor}\label{mythm2}
On any manifold of dimension at least $2$, for any real number $1\leq r<\infty$, there exists a $C^r$-diffeomorphism with finite topological entropy and no measure of maximal entropy.
\end{mycor}

We did not expect the above estimates since consideration of Ruelle's inequality which is typically strict\footnote{The Ruelle inequality is an equality (Pesin formula) iff the measure is of Sinai-Ruelle-Bowen type.} and the following fact (proved in Sect. \ref{sec:proofThm}) would suggest that exponents provide too rough a control:
 
\begin{myfact}\label{rem:bigexpo}
In the setting of Theorem \ref{mythm1} with $1<r<\infty$, 
 $$
    \limsup_{f\to^{C^\infty} f_0} \lambda^u(\Pq(f)) = \lambda(f_0) > \limsup_{f\to^{C^r} f_0} h_\top(f) = \frac{\lambda(f_0)}r.
 $$
\end{myfact}

\bigbreak

Our constructions also apply in the setting of partially hyperbolicity (see, e.g., \cite{BDVbook}).
It has been shown in that setting, that, when the center bundle is one-dimensional (or even a dominated sum of one-dimensional subbundles), there always is a measure of maximal entropy \cite{DFPV2011}. Our examples show that this does not extend to center dimension $2$:

\begin{mycor}\label{coro2}
On the $d$-torus, for any $d\geq4$, there exists a partially hyperbolic diffeomorphism of manifolds with center dimension $2$ and no measure of maximal entropy.
\end{mycor}

\subsection{Comments}

\subsubsection*{Comparison with interval maps}
It is intructive to compare the examples presented here with those built for interval maps (see \cite{Buzzi1997, ruette2002}). They are very similar in being based on a sequence of horseshoes homoclinically related to a given hyperbolic fixed point and with topological entropies converging to, but not achieving the limit $\log\Lambda/r$, where $\Lambda$ is the expanding eigenvalue at the fixed point.

There is however a significant difference which accounts for the relative difficulty (and delay) in building the  diffeomorphisms of this paper. In \cite{Buzzi1997}, there is a sequence of perturbations, each supported by some interval $I_n$, $n\geq1$, and giving rise to a $N_n$-to-$1$ horseshoe which is a forward-invariant disjoint union of intervals $\bigcup_{k=0}^{T_n-1}f^k(I_n)$ and entropy $h_\top(f)=\log N_n/T_n$. The absence of maximal entropy measure follows from the fact that there would be an ergodic one, therefore supported either by a single horseshoe or disjoint from their union. But one can arrange so that $\log N_n/T_n\nearrow \log\Lambda/r$ while $f|\overline{[0,1]\setminus\bigcup_{n,k} f^k(I_n)}$ has smaller topological entropy (the construction in \cite{ruette2002} is subtler, but still relies on an explicit topologically Markov description of the dynamics).

In the construction below, the supports of the perturbations do not define invariant cycles. Indeed, the orbit of their supports accumulate along the unstable manifold of the hyperbolic fixed point, hence they must intersect one another. 

We note that this seems also to prevent a straightforward modification of these examples to construct an example with infinitely many ergodic measures of maximal entropy as was done on the interval by taking $\log N_n/T_n$ constant.

\subsubsection*{Non-differentiable examples}

Ergodic theory sometimes differs greatly between $C^1$ and higher differentiability (see, e.g., Pesin theory or Pugh's closing lemma). However, with regard to existence of measures with maximal entropy, the threshold seems to be more between Lipschitz and non-Lipschitz homeomorphisms. The $C^1$ case is not distinguished, though the analysis of the example simplifies significantly (see Remark \ref{rem:C1}). We note that the Lipschitz case is similar to the case of interval maps. In particular, examples with infinitely many maximal measures are easily obtained (see Appendix \ref{app:nondiff} where folklore examples are described for completeness). For piecewise affine \emph{maps} of compact surfaces, there are examples with no maximal measures \cite{buzziNoMaxAff2009}. For piecewise affine \emph{homeomorphisms} of compact surfaces, Newhouse observed that maximal measures always exist and this author in \cite{buzzi2009pwah} showed that there are finitely many of them.

\subsection{Questions}

\subsubsection*{Does "large entropy" imply existence of a maximal entropy measure?} Indeed, the examples built here satisfy $h_\top(f)\leq\log\Lip(f)/r$. This is known in the setting of interval  maps \cite{buzziruette2006,Burguet2012a}.

\subsubsection*{For diffeomorphisms of $3$-dimensional compact manifolds, does the existence of a dominated splitting imply that of a maximal entropy measure?} We have noted the result of \cite{DFPV2011} which shows that strong partial hyperbolicity (i.e., a splitting $E^s\oplus E^c\oplus E^u$ with no trivial subbundles, hence all of dimension $1$) implies existence. An intermediate question would be: whether (weak) partial hyperbolicity (i.e., a dominated splitting of the type $E^1\oplus E^2$ with $E^1$ uniformly expanding ---or contracting--- and $\dim E^2=2$) implies the existence of a maximal entropy measure.

\subsubsection*{What about infinite multiplicity?} Sarig's construction \cite{sarig2011} of a countable state Markov shift for  $C^{1+\alpha}$ smooth, surface diffeomorphisms shows that, for such maps with positive topological entropy the collection of ergodic measures of maximal entropy is at most countable. However, no examples of $C^r$ diffeomorphisms of compact manifolds with infinitely many ergodic measures with maximal entropy are known for $r\geq1$. We note that it is unknown if a bi-Lipschitz homeomorphism of a compact surface could have \emph{uncountably} many such measures.

\subsection{Outline of the paper and Notations}
We first state the useful properties of the diffeomorphism $f_0$ (Sect. \ref{sec:homomap}). Then we explain the sequence of $C^r$ perturbations used to create entropy for our examples $f$ (Sect. \ref{sec:perturb}). The main point is the control of the Lyapunov exponent of non-quasi-periodic orbits (Sect. \ref{sec:exponent}). We then conclude in Sect. \ref{sec:proofThm} with the proofs of the main results.

Two appendices are provided for the convenience of the reader. Appendix \ref{app:nondiff} describes the situation for Lipschitz homeomorphisms and general homeomorphisms. Appendix \ref{app:construct} gives the rather easy details of the construction of the map $f_0$.

\subsubsection*{Notations} We denote by the same letter $C$ different numbers larger than one and independent from the parameters of our construction: $K,L,n_0$ (see Sect. \ref{sec:homomap} and \ref{sec:perturb}). $\pm C$ denotes a number within $[-C,C]$. We denote by $C(K)$, $C(K,L)$,... different functions of $K$, of $K,L$,... 

$\mathcal O(\cdot)$ and $o(\cdots)$ are the usual Landau notations, the limit here being $K,L,n_0\to\infty$ (possibly under some constraints specified in the context). $\Omega(f(K,L,n_0))$ denotes positive functions which are equal to $f(K,L,n_0)$ up to a bounded, positive factor when $K,L,n_0$ are large enough. 

\section{Homoclinic map $f_0:D\to D$}\label{sec:homomap}

We start from a diffeomorphism on the disk $D=\overline{B(0,2)}\subset\RR^2$ exhibiting a homoclinic loop at a strongly dissipative, hyperbolic fixed point. For convenience, it is defined as a map  $f_0:D\to D$ with four such fixed points, identified by a four-fold symmetry given by: $\tau_0(x,y)=(x,y)$, $\tau_1(x,y)=(y,-x)$, $\tau_2(x,y)=(-x,-y)$, $\tau_3(x,y)=(-y,x)$. The following statement summarizes the properties we shall use (see Figure \ref{fig:homoloop}).

\begin{figure}
\includegraphics[width=8cm]{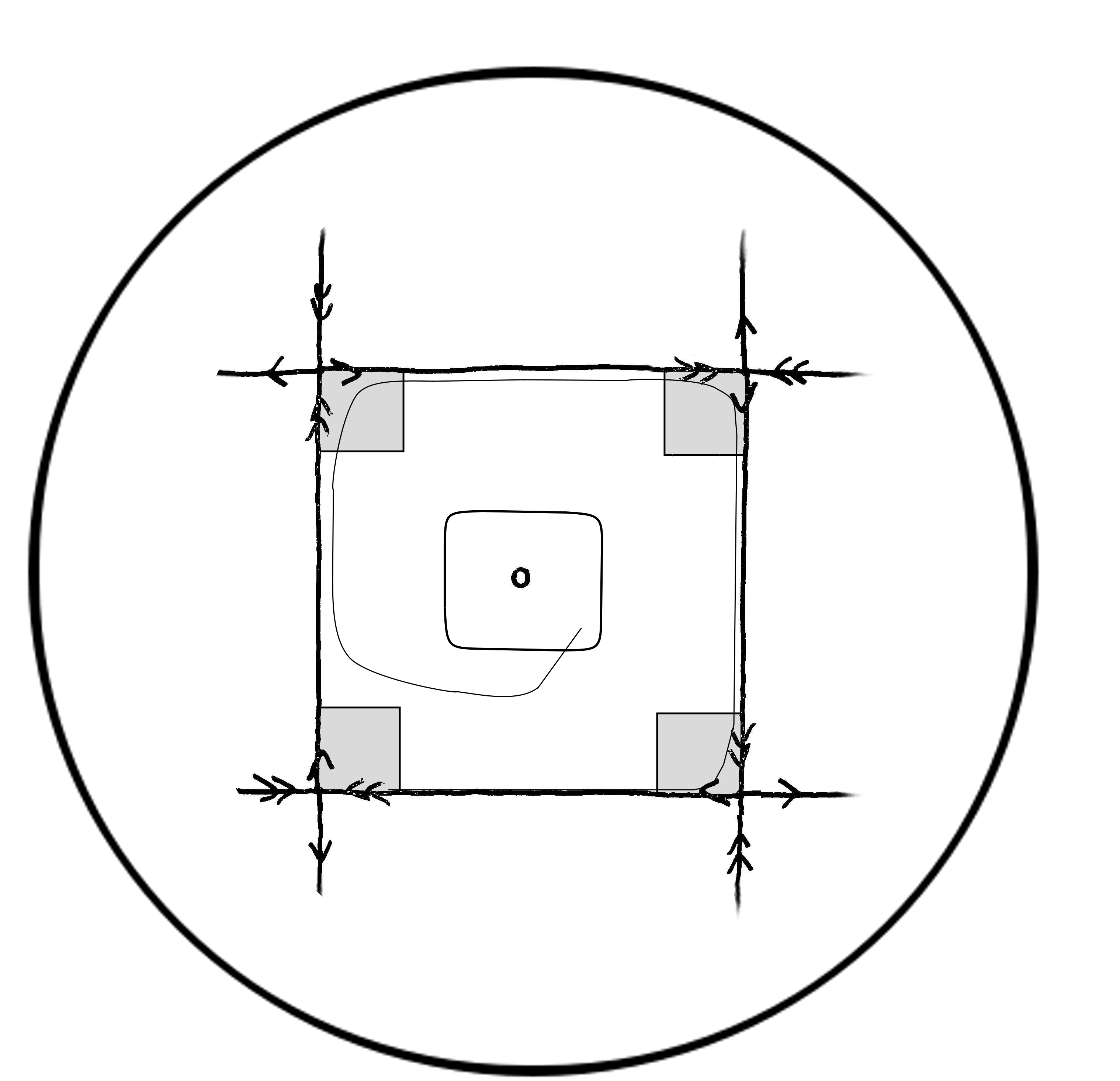}
\caption{The homoclinic map defined inside $D$ with the four dissipative, hyperbolic fixed points at the corners of the square $[-1/2,1/2]^2$ and their heteroclinic connexions and the central, repelling fixed point at $(0,0)$. The four little grey squares are $C$. The centered, smoothed square is $\partial Q$. The spiralling curve is a typical orbit inside $[-1/2,1/2]^2$.}\label{fig:homoloop}
\end{figure}

\begin{prop}\label{prop:factsH}
For all parameters $K,L>C$, there is a $C^\infty$ smooth vector field such that its time $1$ map $f_0:D\to D$ satisfies the following properties. Let $Q:=[-1/6,1/6]^2$ and $C_0:=[-1/2,-5/12]^2$.
\begin{enumerate}
   \item\label{item:nearbound} $f_0$ coincides with the identity near the boundary of $D$, $\lambda(f_0)=\log\Lambda$ and $f_0\circ\tau_i=\tau_i\circ f_0$ for $i=1,2,3$;
  \item\label{item:corner} $f_0|C_0=((x,y)\mapsto(K^{-1}(x+1/2)-1/2,\Lambda (y+1/2)-1/2))$;
 \item\label{item:cycle} All ergodic invariant probability measures are quasi-periodic (i.e., either Dirac measures or with support included on an invariant  topological circle);
 \item\label{item:filter} for all $k\geq C(K,L)$, $f_0^k([-1/2,1/2]^2\setminus Q)\subset [-1/2,1/2]^2\setminus [-5/12,5/12]^2$. Moreover, the omega-limit set of any point in that set coincides with $\partial [-1/2,1/2]^2$.
 \item\label{item:transit} for $x\in f_0(C_0)\setminus C_0$, the transition time to the next affine corner:
  $$
     \tau(x):=\min\{n\geq1:f_0^n(x)\in\tau_{1}(C_0)\}
  $$
satisfies $\tau(x)=c_1L\pm c_2$ where $c_1,c_2$ are two positive constants.
 \item\label{item:transit1} Let $x:=(x_1,x_2)\in f_0(C_0)\setminus C_0$ and $\tau:=\tau(x)$, 
 $$\begin{aligned}
           &f_0^\tau(x): =\left(\tilde\alpha(x_2)(x_1+1/2)-1/2,\tilde h(x_2)\right)\\  &
           \text{ and }T_xf_0^{\tau} = \mat{K^{-\Omega(1)} & -|\mathcal O(K^{-\Omega(1)}(x_1+1/2))| \\ 0 & \Omega(\kappa)}.
  \end{aligned}$$
The implied bounds above can be chosen independently of $K$ and $L$.
 \item\label{item:transit2} There exist a constant $u<1$, independent of $L$ such that for $\kappa\geq 2\lambda$, for all $x\in C_0$,
 \begin{equation}\label{eq:transit1}
     (f_0^\tau(x))_1 \leq u x_1.
 \end{equation}
 \end{enumerate}
\end{prop}

For the sake of definiteness an explicit construction is given in Appendix \ref{app:construct}. It is completely elementary and without difficulty. The only point we would like to stress is that the transition between two visits to the affine neighborhoods of the hyperbolic fixed points cannot exactly map the (previously) expanding direction to the (to be) contracting direction. Nevertheless the angular error is small enough (proportional to the distance to the homoclinic loop) so that it does not spoil the estimates.


\section{Perturbations}\label{sec:perturb}

We are going to perturb the map $f_0:D\to D$ to create a sequence of horseshoes with entropies strictly increasing to $\log\Lambda/r$. The resulting map $f$ will have entropy $h_\top(f)=\log\Lambda/r$ and no measure of maximal entropy. This would follow immediately {\it provided there would be no other sources of entropy}, but we are not able to prove it directly. Fortunately, one can prove a strong enough bound on the exponents, see Sec. \ref{sec:exponent}.

\subsection{Construction}

We will make perturbations inside the four corners $\tau_i(C_0)$, $i=0,1,2,3$,  keeping the four-fold symmetry. We only need to specify the perturbations inside the corner $C_0$. It will be convenient to use coordinates identifying $C_0$ to $[0,2]^2$ and $f_0$ to the linear map defined by $A=\mat{K^{-1}&0\\0&\Lambda}$.

The perturbations are indexed by $n\in\NN$, $n\geq n_0$, with $n_0$ an arbitrarily large integer. The $n$th perturbation will be supported around $I_n\times\{0\}:=[1+1/n^2,1+1/n^2+1/n^4]\times\{0\}$. It will be made of $N_n$ "wiggles" pushing upward the unstable manifold of the fixed point $(1/2,-1/2)$ in the original coordinates by $\Lambda^{-T_n}$. $N_n$ and $T_n$ are two integer parameters ($T_n$ being, up to an additive constant, the return time), both much larger than $n$.

More precisely, the $n$th perturbation is supported in
 $$
    R_n:=[a_n,b_n]\times [-\ell_n/N_n,\ell_n/N_n]
 $$
with $a_n:=1+1/n^2$, $b_n:=a_n+1/n^4$ and $\ell_n:=b_n-a_n=1/n^4$.
On $R_n$, the original map $f_0:(x,y)\mapsto(K^{-1}x,\Lambda y)$ is replaced by $f_n:=f_0\circ g$ with:
 \begin{equation}\label{eq:formula}
    g:(x,y)\mapsto \left(x, y+\alpha_n(x,y)\Lambda^{-T_n}\left\{2+\sin(\pi N_n(x-a_n)/\ell_n)\right\}\right)
 \end{equation}
with cut-off function
 $$
    \alpha_n(x,y)= \alpha\left(N_n(x-a_n)/\ell_n\right)
                            \alpha\left(N_n(b_n-x)/\ell_n\right)
                            \beta\left(N_ny/\ell_n\right)
$$
where $\alpha:\RR\to[0,1]$ $C^\infty$ smooth, non-decreasing such that $\alpha'(x)=o(\alpha^{1-1/r})$, $\alpha(x)=0$ for $x\leq0$ and $\alpha(x)=1$ for $x\geq 1$ and $\beta:\RR\to[0,1]$ $C^\infty$ smooth, such that $\beta(x)=1$ for $|x|\leq1/2$ and $\beta(x)=0$ for $|x|\geq 1$ and monotone on $\RR_-$ and on $\RR_+$.

Finally, we define $f_{\geq n_0}:D\to D$ by $f_{\geq n_0}(x)=f_n(x)$ if $x\in R_n$ for some $n\geq n_0$ and $f_{\geq n_0}(x)=f_0(x)$ otherwise.

\begin{lem}
Assuming $N_n=o(\Lambda^{-T_n/r}/n^4)$, the $C^\infty$ maps $f_{\geq n_0}$ is $C^r$ and satisfies $\lim_{n_0\to\infty}\|f_{\geq n_0}-f_0\|_{C^r}=0$. Also, for $n_0$ large enough:
 \begin{equation}\label{eq:majder}
         \left|\frac{\partial f_{\geq n_0}(x,y)_2}{\partial x}\right| \leq (f_{\geq n_0}(x,y))_2^{1-1/r}.
 \end{equation}
\end{lem}

Inequality \eqref{eq:majder} above is a consequence of the fact that our perturbations push the unstable manifold up (and never down, thanks to the constant '2' in \eqref{eq:formula}).

\begin{proof}
Thanks to the cut-off functions, the functions $f_n$ are $C^\infty$. Observe that the rectangles $R_n$ are pairwise disjoint so that $f_{\geq n_0}$ is well-defined. It also follows that $f_{\geq n_0}$ is $C^r$ smooth if
 \begin{equation}\label{eq:CrCondition}
  \lim_{n\to\infty} \|f_n-id\|_{C^r}= 0.
 \end{equation}
In this case, $\|f_{\geq n_0}-f_0\|_{C^r}=\sup_{n\geq n_0} \|f_n-f_0\|_{C^r}$.

Now, the partial derivatives of $g_n$ of any order $s\leq r$ are bounded by $C(\ell_n^{-1}N_n)^s\Lambda^{-T_n}\leq C(n^4\Lambda^{-T_n/r}N_n)^r$. Thus \eqref{eq:CrCondition} is satisfied as $N_n=o(\Lambda^{T_n/r}/n^4)$.

The same estimate shows that taking $n_0$ large enough makes the perturbation arbitrarily small in the $C^r$ norm.

We turn to the remaining claim \eqref{eq:majder}.
On the one hand, for $(x,y)\in R_n$,  $(g(x,y))_2 \geq \alpha_n(x,y)\Lambda^{-T_n}$. On the other hand,
 $$\begin{aligned}
   \left|\frac{\partial\alpha_n(x,y)}{\partial x}\right|
    &\leq  \frac{N_n}{\ell_n} \times  \left\{\begin{matrix}\alpha'(N_n(x-a_n)/\ell_n)\alpha(N_n(b_n-x)/\ell_n)\beta(N_ny/\ell_n)  \\
                0  \\
                -\alpha(N_n(x-a_n)/\ell_n)\alpha'(N_n(b_n-x)/\ell_n)\beta(N_ny/\ell_n) \end{matrix}\right. \\
   &\leq \frac{N_n}{\ell_n} o(\alpha_n(x,y)^{1-1/r}).
 \end{aligned}$$
where we have used $|\alpha'(u)|= o(\alpha(u)^{1-1/r})$, $|\alpha(u)|, |\beta(u)|\leq 1$ and distinguished three cases depending on the position of $x$ with respect to $a_n+\ell_n/N_n$ and $b_n-\ell_n/N_n$.
Therefore,
 $$\begin{aligned}
    \left|\frac{\partial g(x,y)_2}{\partial x}\right| &\leq \Lambda^{-T_n} \times \left(
        3\left|\frac{\partial\alpha_n(x,y)}{\partial x}\right|
      + \pi\frac{N_n}{\ell_n} \alpha_n(x,y)\right)\\
   &\leq \frac{\Lambda^{-(1-1/r)T_n}}n \left( o(\alpha_n(x,y)^{1-1/r}) + \pi\alpha_n(x,y)\right) \\
   &\leq (g(x,y))_2^{1-1/r}.
 \end{aligned}$$
Finally, observe that $|\partial f_2/\partial x|=\Lambda |\partial g_2/\partial x|$ and $f(x,y)_2=\Lambda g(x,y)_2$.
\end{proof}

\begin{lem}\label{lem:perturb0}
If $(x,y)\in R_n$ and $\tau=\min\{k\geq0:f^k(x,y)\notin C_0\}$, then:
 $$
     f_n'(x,y) = \mat{K^{-1} & 0 \\ \mathcal O(1)\Lambda^{-(1-1/r)\tau} & \Lambda+\mathcal O(\Lambda^{-T_n}) }.
 $$
\end{lem}

\begin{proof}
This follows immediately from \eqref{eq:formula} and especially the estimate \eqref{eq:majder} using $y=C^{\pm1} \Lambda^{-\tau}$.
\end{proof}

\subsection{Entropy of the horseshoes}\label{sec:entropy}

We show that $h_\top(f_n^{T_n}|R_n)\geq \log N_n$ for each $n\geq n_0$, yielding $h_\top(f_{\geq n_0})\geq \log\Lambda/r$. We assume from now on that $N_n:=[\Lambda^{T_n/r}/n^5]$. We begin by exhibiting a \emph{horseshoe} inside $R_n$. $\|\cdot\|_{\sup}$ denotes the supremum norm.

\begin{lem}
For  $j=1,\dots,N_n-1$, let $I_j:=[a_n+(j-1/4)\ell_n/N_n,a_n+(j+1/4)\ell_n/N_n]$ and let $\mathcal G_j$ be the set of graphs of functions $\phi:I_j\to \RR$ satisfying $\|\phi\|_{\sup}<\Lambda^{-T_n-1}/10$ and $\|\phi'\|_{\sup}<K^{-T_n}$. Then, for each couple $j,k=1,\dots,N_n-1$, for any $\Gamma\in\mathcal G_j$, there exists $\Gamma'\in\mathcal G_k$ such that $\tau_{i+1}^{-1}\circ f_n^{T_n}(\tau_i\circ\Gamma)\supset\Gamma'$ for some $i=0,1,2,3$.
\end{lem}

\begin{proof}
We assume $i=0$.
$\Gamma$ is the graph of some function $\phi:I_j\to\RR$. Applying $f_n$ once, the image $(x,y)\in f_n(\Gamma)$ satisfies $|y|\leq 3\Lambda^{-T_n+1}$ and
 $$\begin{aligned}
     \left|\frac{dy}{dx}\right| &\geq (K\Lambda )((\pi/\sqrt{2})\Lambda^{-T_n}N_n/\ell_n -K^{-T_n}) \\
     &\geq (K\Lambda )\Lambda^{-(1-1/r)T_n}/n.
 \end{aligned}$$
 $f^{T_n-1}$ acts linearly on this curve. $f^{T_n}(\Gamma)$ contains a piece which is at a Hausdorff distance at most $CK^{-T_n}$ to the vertical segment $\{-1/2\}\times[0,1/6]$. The slope of this piece is:
 $$\begin{aligned}
   \left|\frac{dy}{dx}\right| & \geq
      (K\Lambda)^{T_n} \Lambda^{-(1-1/r)T_n}/n
     = \Lambda^{T_n/r}K^{T_n}/n.
  \end{aligned}$$
That is, it is almost vertical
 $
    \left|\frac{dx}{dy}\right|  = o(K^{-T_n}).
 $
Applying the (non-linear) transition $f_0^\tau$, which according to Prop. \ref{prop:factsH} and Lemma \ref{item:transit}, has differential $\mat{\mathcal O(1) &  K^{\Omega(1)}x\\ 0 & |\Omega(\log K)|}$, we get a curve $CK^{-T_n}$-close to $[0,1/6]\times\{-1/2\}$ with slope:
 $
     \left|\frac{dx}{dy}\right|  \leq K^{-T_n}.
 $
Such a curve contains the image by $\tau_1$ of some graph $\Gamma'\in\mathcal G_k$, proving the claim.
\end{proof}

Iterating $m$ times the above lemma, one gets at least $(N_n-1)^m$ orbits which are $(\ell_n/N_n/2,mT_n)$-separated. Taking the limit $m\to\infty$, we see that:

\begin{cor}\label{cor:perturb2}
For any $K\geq C$ and $L\geq C$,  for any $\eps>0$, any $n_0\geq C(\eps)$, the map $f=f^{K,L}_{\geq n_0}$ is $C^r$ and satisfies
 $$
    \|f-f_0\|_{C^r}<\eps \text{ and }h_\top(f)\geq \sup_{n\geq n_0} \log(N_n-1)/T_n = \log\Lambda/r.
 $$
\end{cor}

\section{Lyapunov exponent}\label{sec:exponent}

We consider $f=f_{\geq n_0}^{K,L}$ for some parameters $L,K,n_0$ to be specified. Each nonzero vector $(x,v)\in TD$ defines a (lower) exponent:
 $$
     \lambda(x,v) := \liminf_{n\to\infty} \frac1n\log\|Df^n(x).v\|.
 $$
Observe $f(5/12,-1/2)=(2/5,-1/2)$ so that $]2/5,5/12]$ is a fundamental domain of the restriction of $f$ to $]-1/2,1/2[\times\{-1/2\}$. Let $\Delta:=\bigcup_{i=0}^3 \tau_i(]2/5,5/12]\times[-1/2,-1/10]$.

\begin{prop}\label{prop:expo}
For any $K\geq C$, any $L\geq C(K)$,any $n_0\geq C(L)$, the following holds. There is $\chi>0$ such that, for any $x\in]-1/2,1/2[^2\setminus Q$ and all $v\in\RR^2\setminus\{0\}$, 
  \begin{equation}\label{eq:boundexp}
     \lambda(x,v) < \log \Lambda/r - \chi \times \liminf_{n\to\infty} \frac1n\#\{0\leq k<n: f^kx\in\Delta\}.
 \end{equation}
\end{prop}

\begin{rem}\label{rem:C1}
For $r=1$, \eqref{eq:boundexp} follows immediately from $\Lip(f|\Delta)<\Lambda$.
\end{rem}

\subsection{Division of the orbit}

For $x\in]-1/2,1/2|^2\setminus Q$ and $v\in\RR^2\setminus\{0\}$, we are going to show \eqref{eq:boundexp} with $\chi=\log A$ for some $A>1$ and "large" parameters $K,L,n_0$ (how large will be specified).

Recall from Proposition \ref{prop:factsH} that  $C:=\bigcup_{i=0}^3 \tau_i(C_0)$ is the union of the four corners of $[-1/2,1/2]^2$ where $f_0$ is affine. We call \emph{affine segments} the maximal intervals of integers $k\in\ZZ$ such that $f^k(x)$ belongs to $C$. There are infinitely many of them in $\NN$. By replacing $(x,v)$ with some iterate (which does not change $\lambda(x,v)$), we assume that $x\in C$, $f^{-1}(x)\notin C$. 

We denote the affine segments by $S_n:=[t_{n-1},s_n]$, for $n=1,2,\dots$ ($t_0=0$ by the assumption).  Observe that $t_{n-1}$ is the unique $t\in S_n$ for which $f^t(x)$ may belong to the support $R_{p_n}$ of some perturbation (for some $p_n\geq n_0$).  

We also set $\tau_n:=t_n-t_{n-1}$, $\tau_n':=s_n-t_{n-1}$ and $d_n:=t_n-s_n$,  $x(t):=f^t(x)$, $v(t):=(f^t)'(x).v$, $x_n:=f^{t_n}(x)$ and $v_n:=(f^{t_n})'(x).v$.
Let $0\leq \alpha(t)\leq\infty$ be the absolute value of the tangent of the angle between $v(t)$ and $\tau_i'(e_1)$ (if $x_n\in\tau_i(C_0)$) and let $\theta_n:=\alpha(t_n)$ and $\tilde\theta_n:=1/\alpha(s_n)$.

For $t\in[t_n,t_{n+1}[$
 such that $x(t_n)\in \tau_i(C_0)$, we use coordinates:
 $$\begin{aligned}((x(t))_1,(x(t))_2):=\tau_i^{-1}(x(t))+(1/2,1/2)\\
 \text{ and }((v(t))_1,(v(t))_2):=T_{x(t)}\tau_i^{-1}.(v(t)).
 \end{aligned}$$

The lengths of the affine segments cannot grow too fast. More precisely, 

\begin{lem}\label{lem:growth-seg}
Assume that $L\geq C$. Then, for every $n\geq1$, 
 $$
    K^{-\tau'_n}\Lambda^{\tau'_{n+1}}\leq K^{\Omega(1)} \text{ and }
    K^{-\tau_n}\Lambda^{\tau_{n+1}} \leq  K^{-\Omega(L)}\Lambda^{\Omega(L)}.
 $$
\end{lem}

\begin{proof}
Using Proposition \ref{prop:factsH} and $ (x_{n-1})_1\geq C^{-1}$,
 $$\begin{aligned}
   (x_n)_2 &\geq K^{-\Omega(1)} (x(s_n))_1 \geq K^{-\tau_n'-\Omega(1)}(x(t_{n-1}+1))_1 \\
    &\geq  K^{-\tau_n'-\Omega(1)} (x_{n-1})_1 \geq K^{-\tau_n'-\Omega(1)} .
 \end{aligned}$$
As the perturbation only pushes "upward", $(x(t_n+1))_2\geq \Lambda (x_n)_2$ and
 $$\begin{aligned}
  C \geq  (x(s_{n+1}))_2& \geq C^{-1} \Lambda^{\tau_{n+1}'}(x_n)_2 \geq \Lambda^{\tau'_{n+1}}K^{-\tau_n'-\Omega(1)}\\
 \end{aligned}$$
The claim follows as $\tau_m=\tau'_m+d_m$ and $d_{m}=\Omega(L)$ for all $m$.
\end{proof}

\begin{lem}\label{lem:order}
For $n_0\geq C$ and $L\geq C$, the lengths of the affine segments are eventually large:
 $$
    \liminf_{n\to\infty} \tau_n \geq T_{n_0}/r.
 $$
\end{lem}

\begin{proof}
If $x_n\in R_p$ for some $p\geq n_0$, then $(x_n)_2\leq N_p/\ell_p\leq C p^{-1}\Lambda^{-T_p/r}$ and therefore $\tau'_n\geq T_p/r+\log p-C\geq T_{n_0}/r$ ($n_0$ is large).

If $x_n\notin\bigcup_{p\geq n_0}R_p$, then $(x(s_{n+1}))_1= K^{-\tau'_n}(x_{n})_1$ and, using eq.~\eqref{eq:transit1} of Prop.~\ref{prop:factsH}: $(x_{n+1})_2\leq (x(s_{n+1}))_ 1$. Thus, $\tau'_{n+1}\geq \eta\tau'_n$.

In both cases, $\tau'_{n+1}\geq\min(\eta\tau'_n,T_{n_0}/r)$ and therefore $\tau_n\geq \tau'_n\geq T_{n_0}/r$ for all large $n$.
\end{proof}

\begin{lem}\label{lem:angletransition}
We have:
 $$
   \theta_n = K^{-\Omega(1)}\left(\tilde\theta_n+\mathcal O(1)(x(s_n))_1\right)
                 = K^{-\Omega(1)}\left(\tilde\theta_n + \mathcal O(1) \Lambda^{-\tau_{n+1}}\right).
 $$
\end{lem}

Observe that if the transition map was a translation, one would simply have $\tilde\theta_n=\theta_n$.

\begin{proof}
This follows from the transition map. According to Proposition \ref{prop:factsH}, we have:
 $$\begin{aligned}
    (x_n)_1 &= C - \tilde h(x(s_n)_2) \text{ and } (x_n)_2 = \tilde\alpha(x(s_n)_2)x(s_n)_1\\
    (v_n)_1 &=\tilde h'(x(s_n)_2)v(s_n)_2 \text{ and }
              (v_n)_2 = \tilde\alpha(x(s_n)_2)v(s_n)_1 +\\
                &\qquad\qquad\qquad\qquad\qquad\qquad\qquad\qquad + \tilde\alpha'(x(s_n)_2)x(s_n)_1v(s_n)_2.
 \end{aligned}$$
Thus, using $\alpha'(x)=\mathcal O(\alpha(x))$ (see Sec. \ref{sec:proof-H}),
 $$\begin{aligned}
   \theta_n &= \frac{(v_n)_2}{(v_n)_1} = \frac{\tilde\alpha(x(s_n)_2)v(s_n)_1 + \tilde\alpha'(x(s_n)_2)x(s_n)_1v(s_n)_2}{\tilde h'(x(s_n)_2)v(s_n)_2} \\
   &= \frac{\tilde\alpha(x(s_n)_2)}{\tilde h'(x(s_n)_2)}\left(\tilde\theta_n + \mathcal O(1) x(s_n)_1\right)\\
   &= \frac{K^{-\Omega(1)}}{\log K} \left(\tilde\theta_n + \mathcal O(1) (x(s_n))_1\right)
 \end{aligned}$$
The claims follow.
\end{proof}

\subsection{Expansion during one affine segment}

We bound the expansion for different types of affine segments, distinguished according to the value of the tangent $\theta_n$ of the angle between $v_n$ and $e_1$.

\begin{defn}\label{def:specialtime}
A time $t_n$ is said to be \emph{special} if $\theta_n>\Lambda^{-(1-1/r)\tau_{n+1}}$.
\end{defn}

Remark that this definition could have been made in terms of $\tilde\theta_n$, the angle at the end of the affine segment just before $t_n$. Indeed, $\tilde\theta_n>K^{-\Omega(1)}\Lambda^{-(1-1/r)\tau_{n+1}}$ if $t_n$ is special and $\tilde\theta_n<K^{\Omega(1)}\Lambda^{-(1-1/r)\tau_{n+1}}$ otherwise.

\begin{lem}\label{lem:exposeg1}
For any $A<\infty$, for any $L\geq C(K,A)$, and $n_0\geq C(L)$, the following estimate holds. If $t_{n-1}$ is not special, then
 $$
    |v_n| \leq  \Lambda^{\tau_n/r} |v_{n-1}|/A.
 $$
($|\cdot|$ is the Euclidean norm).
\end{lem}

\begin{proof}
Using $f(x_1,x_2)=(K^{-1}x_1,\Lambda x_2)$ in the affine segment, except possibly for $t=t_n$ where we use Lemma \ref{lem:perturb0}, 
 $$\begin{aligned}
    |v(s_n)|
               &\leq C K^{-(\tau'_n-1)}(v(t_{n-1}+1))_1
               +C\Lambda^{\tau'_n-1}(v(t_{n-1}+1))_2\\
             &\leq CK^{-\tau'_n}|v_{n-1}|+C\Lambda^{\tau'_n}
        (\theta_{n-1}+C\Lambda^{-(1-1/r)\tau_n})|v_{n-1}|\\
 \end{aligned}$$
Hence, as $\theta_{n-1}\leq \Lambda^{-(1-1/r)\tau_n}$ and $\tau_n=\tau'_n+\Omega(L)$,
 $$\begin{aligned}
    |v(s_n)|              &\leq (1+2C\Lambda^{\tau_n-\Omega(L)}\Lambda^{-(1-1/r)\tau_n})|v_{n-1}|\\
             &\leq (1+C\Lambda^{-\Omega(L)} \Lambda^{\tau_n/r})|v_{n-1}|\\
             &\leq \Lambda^{-\Omega(L)} \Lambda^{\tau_n/r} |v_{n-1}|
 \end{aligned}$$
($\tau_n/r-\Omega(L)>0$ as $n\geq C(L)$). Thus,
 $$
    |v_n| \leq \mathcal O(\log K)|v(s_n)| \leq \mathcal O(\log K)\Lambda^{-|\Omega(L)|} \Lambda^{\tau_n/r} |v_{n-1}|.
 $$
 To conclude, observe that $\mathcal O(\log K)\Lambda^{-|\Omega(L)|}<1/A$ for $L\geq C(K,A)$.
\end{proof}

\begin{lem}\label{lem:exposeg3}
For any $n_0\geq C$, $K\geq C$, $L\geq C$,  if $t_{n-1}$ is special,
 $$
   |v_{n}| \leq K^{-\Omega(1)}\Lambda^{\mathcal O(L)} |v_{n-2}|
 $$
\end{lem}

\begin{proof}
Observe that
 $$
   (x(s_{n-1}))_1\leq K^{\Omega(1)}(x_{n-1})_2\leq K^{\Omega(1)}\Lambda^{-\tau_n}(x_n)_2
        \leq K^{\Omega(1)}\Lambda^{-\tau_n} .
 $$
Using Lemma \ref{lem:angletransition}  :
 $$\begin{aligned}
   \frac{\tilde\theta_{n-1}}{\theta_{n-1}}
         &\geq K^{-\Omega(1)}-\mathcal O(1)\frac{(x(s_{n-1}))_1}{\theta_{n-1}} 
           \geq K^{-\Omega(1)}-\mathcal O\left(\frac{K^{-\tau_n}}{\Lambda^{-(1-1/r)\tau_n}}\right) \\
         & \geq   K^{-\Omega(1)}-\mathcal O\left(\left(K^{-1}\Lambda^{1-1/r}\right)^{\tau_n}\right)
           \geq K^{-\Omega(1)},
 \end{aligned}$$
as $t_{n-1}$ is special and $\tau_n\geq T_{n_0}/r$ is large.
We compute:
 $$\begin{aligned}
   |v(s_{n})| \leq \Lambda^{\tau'_{n}}\theta_{n-1} |v_{n-1}| 
      \leq \mathcal O(\log K) \Lambda^{\tau_{n}'}\theta_{n-1} |v(s_{n-1})|
 \end{aligned}$$
and
 $$
   |v(s_{n-1})| \leq C \frac{(v(s_{n-1}))_1}{\tilde\theta_{n-1}} = C \frac{K^{-\tau'_{n-1}+1}}{\tilde\theta_{n-1}} (v(t_{n-2}+1))_1.
 $$
Therefore, using $(v(t_{n-2}+1))_1\leq K^{-1}|v_{n-2}|$ and Lemma \ref{lem:growth-seg},
 $$\begin{aligned}
  |v_{n}| & \leq \mathcal O(\log K) |v(s_{n})| 
     \leq  \mathcal O(\log^2 K) \Lambda^{\tau_{n}'}\theta_{n-1} \frac{K^{-\tau'_{n-1}}}{\tilde\theta_{n-1}} |v_{n-2}| \\
      &\leq  \mathcal O(\log^2 K)  \frac{\theta_{n-1}}{\tilde\theta_{n-1}} \Lambda^{\tau_{n}'}K^{-\tau'_{n-1}} |v_{n-2}| \\
     &\leq  \mathcal O(\log^2 K)  K^{-\Omega(1)} \Lambda^{\Omega(L)} |v_{n-2}|.
 \end{aligned}$$
\end{proof}

\subsection{Proof of Proposition \ref{prop:expo}}

We set $A:=e$. We pick $K\geq C$, $L\geq C(K,A)$ and $n_0\geq C(L)$ as required by the constructions in  Corollary \ref{cor:perturb2} and Lemmas \ref{lem:exposeg1} and \ref{lem:exposeg3}.  If necessary we increase the lower bound on $n_0$ to ensure
 $$
    \Lambda^{T_{n_0}/r}>A^2K^{-\Omega(1)}\Lambda^{\mathcal O(L)}.
 $$
Lemma \ref{lem:order} gives $N_1=N_1(x,v)<\infty$ such that $t_{n+1}-t_n\geq T_{n_0}/r$ for all $n\geq N_1$. Let $N\geq N_1$.

We shall use the following regroupings of affine segments for $(x,v)$ chosen as in the statement of the Proposition.

\begin{defn}
 A \emph{special block} is an integer interval $[t_{n-2},t_n[\subset[t_{N_1},t_N[$ such that $t_{n-1}$ is special. A \emph{normal block} is a non-empty integer interval $[t_{n-1},t_n[\subset[t_{N_1},t_N[$ such that $t_{n-1}$ is not special.
\end{defn}

Proposition \ref{prop:expo} will be a consequence of the following claim, if one observes that $N$ is equal to the number of visits to $\Delta$ before time $t_N$.

\begin{cla}\label{claim:expo}
There exists a partition of $[0,t_N[$ or $[t_1,t_N[$ into a disjoint union of special blocks and normal blocks. Moreover, there exists $C(x,v)<\infty$, independent of $N$, such that
 \begin{equation}\label{eq:vN-bound}
     |v_N| \leq C(x,v)\Lambda^{t_N/r}/A^N.
 \end{equation}
\end{cla}
 
\begin{proof}[Proof of the claim]
To define the partition into blocks, let $a_1:=t_N$ and, inductively:
 $$
    a_{k+1}:=\min\{t_j:j\geq t_{N_1}\text{ s.t. }[t_j,a_k[\text{ is a normal or special block}\}.
 $$
We set $I\geq1$ minimum such that $a_{I+1}$ is not defined. For $k=1,\dots,I$,  $j_k$ denote the unique integer such that $a_k=t_{j_k}$. Let us check:
 \begin{equation}\label{subclaim}
   j:=j_I \text{ is $N_1$ or $N_1+1$.}
 \end{equation}
Assume by contradiction $j\geq N_1+2$. If $t_{j-1}$ is special then $[t_{j-2},t_j[\subset[t_{N_1},t_N[$ is a special block and one could set $a_{I+1}=t_{j-2}$. Otherwise $[t_{j-1},t_j[\subset[t_{N_1},t_N[$ is a normal block and one could set $a_{I+1}=t_{j-1}$. In both cases, this contradicts the definition of $j=j_I$, proving \eqref{subclaim}.

Lemma \ref{lem:exposeg1} states that if $[a_{k+1},a_k[$ is a normal block, then $|v(a_k)|\leq (\Lambda^{(a_k-a_{k+1})/r}/A) |v(a_{k+1}|$ and $j_{k}=j_{k+1}+1$. Lemma  \ref{lem:exposeg3} shows that if $[a_{k+1},a_k[$ is a special block then $|v(a_k)|\leq \Lambda^{\mathcal O(L)}K^{-\Omega(1)} |v(a_{k+1})| \leq (\Lambda^{(a_k-a_{k+1})/r}/A^2) |v(a_{k+1})|$ as $a_k-a_{k+1}\geq 2T_{n_0}/r$ is large, thanks to $n_0\geq n_0(K)$. Also $j_k=j_{k+1}+2$ in this case.

We get:
 $$\begin{aligned}
     |v_N| &\leq C(x,v) \prod_{i=1}^I \frac{\Lambda^{(a_{i+1}-a_i)/r}}{A^{j_{i+1}-j_i}}
        \times \frac{\Lambda^{a_I/r}}{A^I}
       \leq C(x,v) \Lambda^{t_N/r}/A^N
 \end{aligned}$$
with $C(x,v):=|v_{a(I)}|\Lambda^{-a_I/r}A^I$. \eqref{eq:vN-bound} is established.
\end{proof}

\section{Proof of the Main Results}\label{sec:proofThm}

\begin{proof}[Proof of Theorem \ref{mythm1}]
Let $\eps>0$. We fix the parameters $K\geq C$, $L\geq C(K)$, $n_0\geq C(L,\eps)$ as in Corollary \ref{cor:perturb2} and Proposition \ref{prop:expo} 
and set $f:=f^{K,L}_{\geq n_0}$, $f_0:=f^{K,L}$. By the choice of $n_0$, we have $\|f-f_0\|_{C^r}<\eps$.

Let $\mu\in\Pq(f)$.  Observe that $\mu$ cannot satisfy $\mu(]-1/2,1/2[^2\setminus Q)=0$ as it would be an invariant measure for $f_0$ and therefore be quasi-periodic according to item \eqref{item:cycle} of Proposition \ref{prop:factsH}. By ergodicity and the invariance from item \eqref{item:filter} of Proposition \ref{prop:factsH}, $\mu(]-1/2,1/2[^2\setminus Q)=1$. Thus, $\mu$-a.e. orbit visits the set $\Delta$ defined before Proposition \ref{prop:expo},  implying $\mu(\Delta)>0$.

According to Proposition \ref{prop:expo} (and the Birkhoff theorem applied to the characteristic function of $\Delta$), $\lambda(x,v)<\log\Lambda/r$ for $\mu$-a.e. $x$ and all nonzero $v\in\RR^2$. In particular, the two Lyapunov exponents of $\mu$ are themselves strictly less than $\log\Lambda/r$. Using Ruelle's inequality, $h(f,\mu)\leq\lambda(f,\mu)^+<\log\Lambda/r$. The variational principle implies
 $$
     h_\top(f)=\sup_{\mu\in\Pq(f)} h(f,\mu) \leq \log\Lambda/r,
 $$
as quasi-periodic measures have zero entropy. Comparing with Corollary \ref{cor:perturb2}, we get: $h_\top(f)= \log\Lambda/r$.

The same variational principle gives the inequality $\lambda^u(\Pq(f))\geq h_\top(f)$ so $\lambda^u(\Pq(f))=\log\Lambda/r$. The theorem is proved.
\end{proof}

\begin{proof}[Proof of Corollary \ref{mythm2}]
Let $M$ be any manifold of dimension $d\geq2$. Let $K:=D\times\TT^{d-2}$ (recall $D$ is a closed disk in $\RR^2$) and $F(x,t)=(f(x),t)$. Observe that $\partial K=\partial D\times\TT^{d-2}$, so $F$ is a $C^r$ diffeomorphism coinciding with the identity near the boundary of $K$. $M$ contains a $d$-dimensional ball and therefore a diffeomorphic copy of $K$. Define  a $C^r$ diffeomorphism $G$ by setting $G|M\setminus K\equiv Id$ and $G|K=F$. $h_\top(G)=h_\top(F)>0$ according to the variational principle for ergodic measures.

Observe that any ergodic measure with maximal entropy for $G$ must be carried by $K$. But then its projection by $K\to D$ is a measure for $f$ with the same entropy, hence of maximal entropy for $f$, a contradiction: $G$ has no measure of maximal entropy.
\end{proof}

\begin{proof}[Proof of Remark \ref{rem:bigexpo}]
Observe that $\lambda(f)=\inf_{n\geq1}\log\Lip(f^n)$, hence $f\mapsto \lambda(f)$ is upper semicontinuous: $\limsup_{f\to f_0} \lambda(f)\leq \lambda(f_0)$. Thus,
 $$
     \limsup_{f\to^{C^r} f_0} h_\top(f)\leq \lambda(f_0)/r
 $$
is a corollary of Yomdin's theory \cite{Yomdin1987} using $h_\top(f_0)=0$. The reverse inequality follows from Theorem \ref{mythm1}. The first claim in Remark \ref{rem:bigexpo}: $ \limsup_{f\to^{C^r} f_0} h_\top(f)= \lambda(f_0)/r$ is proved.

For the second claim, observe that  Lyapunov exponents are bounded by the logarithm of the Lipschitz constant, hence the following direction of the inequality is obvious: $$\limsup_{f\to^{C^1} f_0} \lambda^u(\Pq(f))\leq\lambda(f_0).$$ 
To conclude, we exhibit $C^\infty$ perturbations $f$ of $f_0$ with $\lambda^u(\Pq(f))\geq \lambda:=\lambda(f_0)$. Let us sketch their construction.

To build such diffeomorphisms $f$,  we make a perturbation around $[a_n,b_n]\times\{0\}$ as in Sect. \ref{sec:perturb} (recall $a_n=1+1/n^2$, $b_n=a_n+1/n^4$). We need to remain $\Lambda^{-T_n}$-close to the axis to get a large period $T_n$. To obtain a large exponent however, we need a relatively large slope. We get it by using transverse intersections of the (perturbed) unstable manifold with the (intact) stable manifold. We replace the map $g$ of \eqref{eq:formula} from Sect. \ref{sec:perturb}, by
 $$\begin{aligned}
    \bar g(x,y) &= (x,y+\phi_n(x,y))\\ &\text{ where } \phi_n(x,y):=\bar \alpha_n(x,y) e^{-(\log n)^2} \cos(10 \pi n^4 x)
 \end{aligned}$$
with the cut-off $\bar\alpha_n(x,y)=\alpha(10n^4(x-a_n))\alpha(10n^4(b_n-x))\beta(n^4y)$ (using $\alpha$ and $\beta$ from Sect. \ref{sec:perturb}). We thus obtain maps $\bar f_n$ and $\bar f_{\geq n}$.

Observe that $\|f_0-\bar f_n\|_{C^r}\leq (C n)^{4r}e^{-(\log n)^2}$ for any fixed, finite $r$. It follows that $\bar f_{\geq n}$ are $C^\infty$ maps   and $\lim_{n\to\infty} \|f_0-\bar f_{\geq n}\|_{C^r}=0$.

Consider a small neighborhood $U_n$ of the graph of $x\mapsto \phi_n(x,0)$ restricted to the values $x\in[a_n,b_n]$ such that $\phi_n(x,0)\in[(5/12)\Lambda^{-T_n},\Lambda^{-T_n}]$. The absolute value of the slope there is at least $n^4e^{-(\log n)^2}\Lambda^{-T_n}$. The first return time to $U_n$ is $T_n\pm C$. The intersection of $U_n$ with $f^{T_n\pm C}$ is easily seen to be hyperbolic. We thus obtain a horseshoe as the maximal invariant set for $U_n$ under $f^{T_n\pm C}$.

It is routine to check that the Lyapunov exponent of any ergodic invariant measure supported by this horseshoe is at least
 $$
  \frac1{T_n+C}(T_n\lambda -(\log n)^2+4\log n-C)\to \lambda,
 $$
as $T_n\geq n$ (indeed, $T_n>>n$). Thus $\lambda^u(\Pq(f_{\geq n}))=\lambda$ for any $n$.
\end{proof}

\begin{proof}[Proof of Corollary \ref{coro2}]
Let $d\geq2$. Theorem \ref{mythm1} gives a $C^r$ diffeomorphism $f:\TT^2\to\TT^2$ with no maximal measure. Let $\phi:\TT^{d-2}\to\TT^{d-2}$ be a hyperbolic automorphism. Set $\Lambda^u:=\min\{\Lambda\in sp(\phi): |\Lambda|>1\}$ and $\Lambda^s:=\max\{\Lambda\in sp(\phi): |\Lambda|<1\}$. Let $N\geq1$ be a large enough integer such that $(\Lambda^u)^N>\Lip(f)$ and $(\Lambda^s)^N<\Lip(f^{-1})^{-1}$. Define $F:\TT^4\to\TT^4$ by $F(x,y)=(\phi^N(x),f(y))$. It is easy to check that $F$ is partially hyperbolic.

If there was a measure $\mu$ with maximal entropy for $F=\phi^N\times f$, its projections $\pi_1\mu$, $\pi_2\mu$ would satisfy $h(F,\mu)\leq h(\phi^N,\pi_1\mu)+h(f,\pi_2\mu)$. As $h_\top(F)=h_\top(\phi^N)+h_\top(f)$, the projections would have to be measure with maximal entropy for $\phi^N$ and $f$. Hence $F$ like $f$ has no measure of maximal entropy.
\end{proof}

\appendix

\section{Non-differentiable examples}\label{app:nondiff}

\subsection{Lipschitz examples}
Recall that a bi-Lipschitz transformation $h:X\to X$ of a metric space $X$ is a bijection which, together with its inverse, is Lipschitz. The bi-Lipschitz constant is $bi-Lip(h):=Lip(h)+Lip(h^{-1})$ where $Lip(h):=\sup_{x\ne y} d(h(x),h(y))/d(x,y)$.Things are much easier in this category:

\begin{thm}[Folklore]
On any compact surface $M$, there is:
 \begin{itemize}
  \item a bi-Lipschitz transformation $h_0:M\to M$ with no measure of maximal entropy;
  \item a bi-Lipschitz transformation $h_\infty:M\to M$ with (countably) infinitely many ergodic measures of maximal entropy.
 \end{itemize}
\end{thm}

\begin{proof}
We first give construction of $h_0$. There is a family of bi-lipschitz transformations $T_n$ of the unit disk $D$ satisfying:
 \begin{itemize}
  \item $T_n$ is the identity outside of the disk of radius $1/2$;
  \item $bi-Lip(T_n)\leq 10$;
  \item $h_\top(T_1)<h_\top(T_2)<\dots<1$.
 \end{itemize}
Observe also that, for any bi-Lipschitz transformation $f:D\to D$ and any number $\rho>0$, $f_\rho$ the self-map of $D_\rho:=\{x\in\RR^2:|x|<r\}$ defined by $f_\rho(x):=\rho f(x/\rho)$, satisfies $bi-Lip(f_\rho)=bi-Lip(f)$.

Let $\chi:U\to V$, $U\subset \RR^2, V\subset M$,  be some bi-Lipschitz chart of $M$. Let $D_1,D_2,\dots$ be an infinite sequence of pairwise disjoint disks included in $U$, with $D_i=B(x_i,\rho_i)$. The example $f:M\to M$ is defined as follows:
 \begin{itemize}
  \item $f(x)=x$ for $x\in M':=M\setminus \bigcup_{i\geq1} \Delta_i$ with $\Delta_i:=\chi(D_i)\subset M$;
  \item $f(x) =\chi_i\circ T_i\circ\chi_i^{-1}(x)$ if $x\in\Delta_i$, setting $\chi_i:D\to\Delta_i$, $u\mapsto x_i+\rho_i u$.
 \end{itemize}

Observe that each disk $\Delta_i$ is invariant. Hence any ergodic invariant probability measure satisfies exactly one of the following condition: $\mu(\Delta_i)=1$ for some $i\geq1$ or $\mu(M')=1$. In the latter case, $\mu$ is the Dirac measure at some point and $h(f,\mu)=0$. In the second case, $h(f,\mu)\leq h_\top(T_i)$. The variational principle implies that $h_\top(f)=\sup_{i\geq1} h_\top(T_i)$. We see that $h(f,\mu)<h_\top(f)$ by the same reasoning: $f$ has no maximal entropy measure.

To see that $f:M\to M$ is bi-Lipschitz, it is enough to prove it for $F:V\to V$, $F=\chi^{-1}\circ f\circ\chi$, with the canonical Euclidean metric on $V\subset \RR^2$. Let $x,y\in V$, $x\ne y$. We distinguish four cases:

\medbreak\noindent
\emph{Case 1:} $x,y\notin \bigcup_{i\geq1} B(x_i,\rho_i/2)$. Then $fx=f^{-1}x=x$ and $fy=f^{-1}y=y$ and $d(fx,fy)=d(x,y)$.
 
\medbreak\noindent
\emph{Case 2:} $x\in B(x_i,\rho_i/2)$ and $y\in B(x_j,\rho_j/2)$ for some $i\ne j$. We compute $d(fx,fy)\leq \rho_i+\rho_j+d(x,y)$ and $d(x,y)\geq(\rho_i+\rho_j)/2$ so:
 $
    d(fx,fy)/d(x,y)\leq 4 + 1 \leq 5.
 $

\medbreak\noindent
\emph{Case 3:} $x\in B(x_i,\rho_i/2)$ and $y\in B(x_i,\rho_i)$ for some $i\geq1$. Use $bi-Lip(F|B(x_i,\rho_i))=bi-Lip(T_i)\leq 10$.

\medbreak\noindent
\emph{Case 4:} $x\in B(x_i,\rho_i/2)$ and  $y\notin\bigcup_{j\geq1} B(x_j,\rho_j/2)\cup B(x_i,\rho_i)$. We compute $d(fx,fy)\leq \rho_i+d(x,y)$ and $d(x,y)\geq\rho_i/2$ so $d(fx,fy)\leq 5d(x,y)$.

\medbreak

The same reasoning applies to $\Lip(f^{-1})$.

This concludes the construction of $h_0:=f$ with no measure of maximal entropy. To build a bi-Lipschitz transformation $h_\infty$ with infinitely many ergodic invariant probability measures with maximal entropy, one may repeat the previous construction but using a sequence of maps such that $h_\top(T_n)>0$ is constant, instead of strictly increasing (for instance by using the same map infinitely many times).
\end{proof}

\subsection{Homeomorphisms} There exist topologically minimal homeomorphisms of the sphere $S^2$ with non-zero topological entropy and either no measure of maximal entropy, or infinitely many of them, both countably  and uncountably many. This follows from a result of B\'eguin, Crovisier and Le Roux \cite{begin2007construction} and a choice of constructions of subshifts.

Indeed, according to  \cite{begin2007construction}, given any bi-measurable transformation $T$ of a Cantor set and an irrational rotation $R:S^1\to S^1$, there exists a homeomorphism $h$ on the sphere $S^2$ and a bi-measurable conjugacy of restrictions of $h$ and $T\times R$ to subsets of full measure with respect to any invariant probability measure. Well-known examples among subshifts of systems with non-zero topological entropy and infinitely many ergodic, invariant probability measures with maximal entropy (both countably and uncountably many) can therefore be translated on the sphere.

\section{Construction of the homoclinic map $f_0$}\label{app:construct}

We build a $C^\infty$ diffeomorphism $H:=f_0$ of the $2$-disk $D$ satisfying the properties stated in Proposition \ref{prop:factsH} depending on two large parameters $K$ and $L$ specifying the contraction and the time delay respectively.

We set $\Lambda:=6/5$, $\lambda:=\log\Lambda$, $\kappa:=\log K$ and let $\psi:\RR\to[0,1]$ be a $C^\infty$ {bump function}: $\psi(t)=1$ iff $t\leq0$, $\psi(t)=0$ iff $t\geq1$,  and $\psi'(t)\leq0$ for all $t\in\RR$. Also $0<\psi(t)<1$ and $\psi'(t)<0$ for $0<t<1$.

\begin{figure}
\includegraphics[width=8cm]{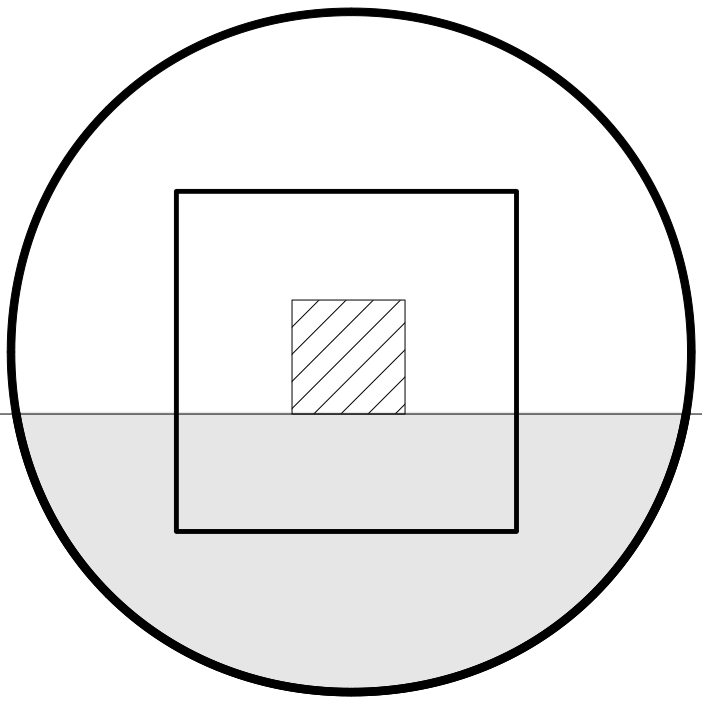}
\caption{The disk is $D:=\overline{B(0,2)}$. The large square is $[-1/2,1/2]^2$, the small, hatched square is $[-1/6,1/6]^2$, and the grey area is $D\cap\{(x,y):y\leq-1/6\}$ the domain of $F_0$.}\label{fig:domaines1.jpg}
\end{figure}

\subsection{Vector field}\label{sec:V0}
We build the homoclinic map $H$ outside $Q:=[-1/6,1/6]^2$ as the time $1$ of a vector field $V:D\setminus Q_0\to \RR^2$ with $Q_0:=[-1/10,1/10]^2$.

\subsubsection*{$V_0$ on $D\cap\{y\leq -1/10\}$}
We first define a vector field $V_0$ on $[-1/2,1/2]\times[-1/2,-1/10]$:
 \begin{equation}\label{eq:V0}
   V_0(x,y) = \alpha_L(x,y) (X_0(x),A_0(x)y)
\end{equation}
with $\alpha_L(x,y)$ a factor restricting the support to $D$ and slowing down the field on $\{\frac1{11}\leq x\leq\frac1{10}\}$ and:
 $$
    X_0(x)=\small \left\{\begin{array}{ll}
                 -\kappa(x+\frac12) & (x\leq \frac1{13})\\
                 \tiny -\psi(156x-12)\kappa(x+\frac12)+\\
                 \qquad - (1-\psi(156x-12))\lambda(\frac12-x)\small & (\frac1{13}\leq x\leq\frac1{12})\\
                 -\lambda(\frac12-x) & (x\geq\frac1{12})
    \end{array}\right.
 $$
and
 $$
    A_0(x)=\small \left\{\begin{array}{ll}
                 \lambda & (x\leq-\frac1{11})\\
                 \tiny\psi(132x+12)\lambda+\\
                 \qquad -(1-\psi(132x+12))\kappa\small & (-\frac1{11}\leq x\leq-\frac1{12})\\
                 -\kappa & (x\geq-\frac1{12}).
    \end{array}\right.
 $$
Finally,
$$
   \alpha_L(x,y)=(L^{-1}+(1-L^{-1})\psi(|440x-42|-1))\psi\left(\frac27(x^2+y^2)-\frac17\right)
 $$
We define $H_0:]-1/2,1/2[\times]-1/2,-1/6[\to]-1/2,1/2[\times]-1/2,-1/10[$ to be the time $1$ map of $V_0$. Observe that is is well-defined as  as $\Lambda\times (1/2-1/6)-1/2=-1/10$.

The non-linearities appears in three vertical strips, namely, from right to left (see Fig. \ref{fig:nonlinear}):
 \begin{itemize}
   \item  $S_1:=\{\frac1{11}\leq x\leq\frac1{10}\}$: the slowing down;
   \item  $S_3:=\{\frac1{13}\leq x\leq\frac1{12}\}$: the transition from $X_0(x)=-\lambda(\frac12-x)$ to $-\kappa(x-\frac12)$.
   \item $S_5:=\{-\frac1{11}\leq x\leq -\frac{1}{12}\}$: the transition from $A_0(x)=-\kappa$ to $\lambda$.
 \end{itemize}
For convenience, we set: $S_0:=\{x\geq\frac1{10}\}$, $S_2:=\{\frac1{12}\leq x\leq\frac{1}{11}\}$, $S_4:=\{-\frac1{12}\leq x\leq\frac1{13}\}$ and $S_6:=\{x\leq-\frac1{11}\}$. We also define maps $h:]-1/2,1/2[\to]-1/2,1/2[$ and $h_x:]-1/2,-1/6[\to]-1/2,-1/10[$ according to $(h(x),h_x(y)):=(x(1),y(1))$ where $t\mapsto(x(t),y(t))$ is the orbit of $V_0$ starting at $(x,y)\in]-1/2,1/2[\times]-1/2,-1/6[$ at time $0$.

\subsubsection*{$V_1$ on $D\setminus Q_0$}
We extend $V_0$ to $D\setminus Q_0$ using the symmetries:
  \begin{equation}\label{defineV}
   V_1(x,y) = \left\{\begin{array}{ll}
              V_0(x,y) & \text{ if }y\leq \Lambda/3-1/2=-1/10\\
              D\tau_1(V_0(\tau_1^{-1}(x,y))) & \text{ if }x\leq \Lambda/3-1/2\\
              D\tau_2(V_0(\tau_2^{-1}(x,y))) & \text{ if }y\geq -\Lambda/3+1/2\\
              D\tau_3(V_0(\tau_3^{-1}(x,y))) & \text{ if }x\geq -\Lambda/3+1/2
              \end{array}\right.
 \end{equation}
where the $\tau_i$'s are the rotations defined before Proposition \ref{prop:factsH}. Observe that, in the above definition of $V_1$, there are four regions on which two formulas are given: for instance, $V_0(x,y)$ and $D\tau_1(V_0(\tau_1^{-1}(x,y)))$ on the square: $(x,y)\in[-2,\Lambda/3-1/2]\times[-\Lambda/3+1/2,2]$, but both are equal to $V_0(x,y)= \psi(\frac27(x^2+y^2)-\frac17)(-\lambda(1-x),-\kappa y)$. $V_1$ is properly defined as a $C^\infty$ vector field on $D\setminus Q_0$. We denote by $H_1:D\setminus Q\to D\setminus Q_0$ with $Q$,  the time $1$ flow of $V_1$ (it is the symmetric extension of $H_0$).

\begin{figure}
\includegraphics[width=8cm]{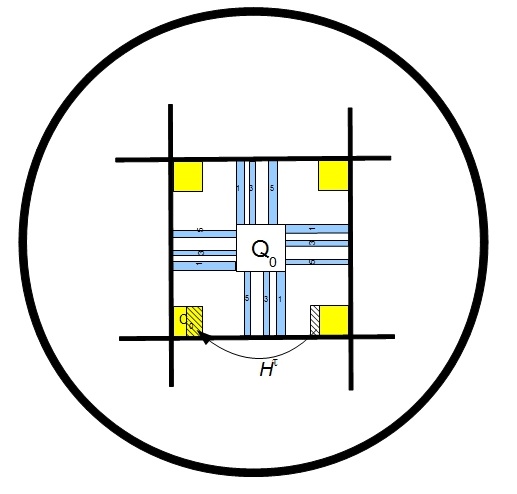}
\caption{The three strips $S_1,S_3,S_5$ of non linearity (marked $1,3,5$ and repeated according to the symmetry). The square $Q_0$, the four corners $\bigcup_{i=0}^3 \tau_i(C_0)$ and hatched, the domain and image of the transition map $H^\tau$ along the bottom side.}\label{fig:nonlinear}
\end{figure}

\begin{lem}\label{lem:flowtobound}
For $K\geq C$ and any $L\geq 1$, splitting $[-1/2,1/2]\times[-1/2,-1/6]$ into a right part and a left part:
 $$\begin{aligned}
        &A_L:=[-1/2,1/11]\times[-1/2,-1/6],
        \; A_R:=[1/11,1/2]\times[-1/2,-1/6],
 \end{aligned}$$
we get,
 \begin{align}
  \label{inclusion1}  (x,y)\in A_L &\implies H_1(x,y)\in [-\frac12,-\frac5{12}]\times[-\frac12,0] \\
  \label{inclusion2}  (x,y)\in A_R  &\implies 
            \alter{\text{either }H_1(x,y)\in[-\frac12,\frac12]\times[-\frac12,-\frac5{12}]  \\
                     \text{or }H_1(x,y)\in[-\frac1{12},\frac1{11}]\times[-\frac12,y[\\
                     \text{or }\alter{\frac1{11}\leq h(x) < x-C^{-1}/L \\ \text{ and } h_x(y)\leq -1/2+K^{-\frac{9}{10}L^{-1}} (y+1/2)}}
 \end{align}
\end{lem}


\begin{proof}[Proof of Lemma \ref{lem:flowtobound}]
We use the maps $h,h_x$ from above. We define $\tau_0:=0$ and, for $i=1,\dots,6$, $\tau_i:=\sup\{t\geq0:x(t+\tau_{i-1})\in S_i\}\cup\{0\}$, the time spent in $S_i$.  Observe that $X_0(x,y)<0$ for $-1/2<x<1/2$.

The first inclusion \eqref{inclusion1} will be a consequence of the claim:
 \begin{equation}\label{eq:claim1}
    h\left(\frac1{11}\right)<-\frac5{12}.
 \end{equation}
(Note that $h_x(-1/6)\leq -1/10<0$.)

Observe that for $x(0)=\frac1{11}$, $\tau_0=\tau_1=0$,  
 $$
   \tau_2=\frac{1}{\lambda} \log\frac{5/12}{9/22}=\frac{1}{\lambda} \log\frac{55}{54}<0.11 \text{ and }    
   \tau_3\leq\frac{1/156}{(5/12)\lambda}<0.09
 $$
(as $|X_0|\geq\lambda(1/2-1/12)=(5/12)\lambda$ on $S_3$). Hence 
 $$
    h(1/11) \leq -1/2+(1/13+1/2)K^{-(\tau_4+\tau_5+\tau_6)}
 $$
and $\tau_4+\tau_5+\tau_6\geq 1-0.20$ so $h(\frac1{11})\leq -\frac12+\frac{7}{13} K^{-0.80}$ from which the claim follows  (for $K$ large enough).

The second inclusion \eqref{inclusion2} will be a consequence of the following claim (for $K$ large enough). For any $(x,y)\in A_R$, one of the following holds:
 \begin{equation}\label{eq:claim2}\begin{aligned}
   &(a)\quad h_x(y)\leq -5/12\\
   &(b)\quad  h_x(y)\leq -1/2+(y+1/2)K^{-\frac{9}{10}L^{-1}}<y\text{ and }\frac{1}{11}\leq h(x)<x-\frac{C^{-1}}L\\ 
   &(c)\quad  h_x(y)<y\leq -1/6 \text{ and } h(x)\in[-1/12,1/11].
 \end{aligned}\end{equation}
Observe:
  $$
   h_x(y)\leq -1/2+(y+1/2)K^{-\tau_0-\tau_1/L-\tau_2-\tau_3-\tau_4}\Lambda^{\tau_5+\tau_6}
  $$
and consider the following cases.

\medbreak\noindent{\sl Case 1:} $\tau_0+\tau_2+\tau_3+\tau_4\geq1/10$. Thus, $\tau_5+\tau_6\leq9/10$ and $h_x(y)\leq-1/2+(1/3)K^{-1/10}\Lambda^{9/10}<-5/12$. (a) is satisfied.

\medbreak
In the remaining cases, we assume $\tau_0+\tau_2+\tau_3+\tau_4<1/10$.
 
\medbreak\noindent{\sl Case 2:} $\tau_0+\tau_1=1$. Thus, $h_x(y)\leq -1/2+(y+1/2)K^{-\frac{9}{10}L^{-1}}<y\leq-1/6$ and $h(x)<x-C^{-1}/L$. (b) is satisfied.

\medbreak\noindent{\sl Case 3:} $\tau_2>0,\tau_5+\tau_6=0$. Thus,  $h(x)\in[-1/12,1/11]$ as $\tau_2>0$ and $\tau_5=0$. As in the previous case, $h_x(y)<y\leq-1/6$. (c) is satisfied

\medbreak\noindent{\sl Case 4:}  $\tau_5+\tau_6>0$, As $x(0)=1/11$, $\tau_2=\frac{1}{\lambda} \log\frac{55}{54}>0.10$ (see above) and therefore:
 $$
    h_x(y)\leq  -1/2+(y+1/2) K^{-0.10}\Lambda^{0.90}<-5/12,
 $$
and (a) is satisfied.
\end{proof}

\subsection{Smooth map $H$}\label{sec:proof-H}

We extend $V_1$ to a $C^\infty$ smooth field $V$ over $D$ such that $V|Q_1=V_1|Q_1$ and let $H:D\to D$ be the time $1$ map of $V$. If necessary, we multiply the vector field by a smooth function less than one over a neighborhood of $Q\cup(D\setminus[-7/12,7/12]^2)$ so that the Lipschitz constant of $H$ over that set is less than $\Lambda$.

\medbreak

\begin{proof}[Proof of Proposition \ref{prop:factsH}]
The first two items are clear from the construction, except for the bound on $\lambda(H)$ which will be proved at the end.

Item \eqref{item:cycle} follows from applying the Poincar\'e-Bendixson theorem to the flow defined by $V$.

The first claim of item \eqref{item:filter} follows from Lemma \ref{lem:flowtobound} for $V$, its invariance under $\tau_i$, $i=0,1,2,3$ and an obvious induction. The statement about the omega limit set follows from the fact that any $V$-orbit must visit the corners infinitely many times in their cyclic order and the estimate \eqref{item:transit2} proved below.

The bound on the transition time $\tau$ in item \eqref{item:transit} follows from the fact that $-X_0(x)\geq c'/L$ for $-5/12\leq x\leq 1/2-(1/12)\Lambda^{-1}=31/72$  and $ -X_0(x,y)\leq c''/L$ for $41/440\leq x\leq 43/440$, for  some constants $c',c''>0$.

We prove  item \eqref{item:transit1} by a computation.  We first  assume that $L=1$ and use the notation $h,h_x$ from section \ref{sec:V0}. $\tilde h(x_2)=h^\tau(x_2)$ with $\tau=\tau(x_1,x_2)=\tau(x_2)$. Using \eqref{eq:V0}, $$h_{x_2}(x_1)=\exp\left(\int_0^\tau A_0(x(s))\, ds\right) (x_1+1/2)-1/2,$$ where $t\mapsto x(t)$ is the solution of $x'(t)=X_0(x(t))$ starting at $x(0)=x_2$. The first claim of item \eqref{item:transit1} follows with $\tilde\alpha(x_2):=\exp\left(\int_0^\tau A_0(x(s))\, ds\right)$.

We deduce the second claim of  item \eqref{item:transit1}. The lower, right entry of the matrix is: $$(h^\tau)'(x_2)=X_0(h^\tau(x_2))/X_0(x_2)=C^{\pm1}\kappa.$$
We claim that the top, left entry is:
  $$\begin{aligned}
        \tilde\alpha(x_2)&=\exp\left(\int_0^\tau A_0(x(s))\, ds\right) =K^{-\Omega(1)}.
  \end{aligned}$$
Indeed, $A_0(x)=-\log K$ for $x\in[1/12,2/5]$ and this occurs for a constant time. Now,  everywhere $A_0(x)\leq \log \Lambda$ and the total time $\tau$ is bounded (assuming, as we may in this estimate, that $L=1$): $\tilde\alpha(x_2)\leq K^{-\Omega(1)}$. $A_0(x)\geq-\log K$ everywhere similarly implies $\tilde\alpha(x_2)\geq K^{-\Omega(1)}$. 

The top, right entry is:
 $$\begin{aligned}
      \tilde\alpha'(x_2)(x_1+1/2) &= \tilde\alpha(x_2) \exp\left(\int_0^\tau A_0'(x(s)) \frac{\partial x(s)}{\partial x(0)}\, ds\right)(x_1+1/2) \\
        &\leq C (x_1+1/2) K^{-\Omega(1)}
 \end{aligned}$$
where $\frac{\partial x(s)}{\partial x(0)}$ is the derivative of $x(s)$ with respect to the initial condition at time $0$. The last inequality uses that $A_0'(x)=132\psi'(123x+12)(\lambda+\kappa)\leq0$ and $\partial x(s)/\partial x(0)>0$: the above exponential is less than $1$.

The remaining entry is clearly zero.

\medbreak

We now extend the previous estimate to arbitrary values of $L\geq1$. We express the transition map  in terms of the passage map $P:\tau_i(C)\to\tau_{i+1}(\{x_*\}\times[0,1])$, $(x,y)\mapsto\Phi^{\tau_L^*(x,y)}(x,y)$ for the flow for parameter $L$. Notice that this passage map does not depend on $L$, though $\tau_L^*$ does (changing $L$ only changes the speed at which the trajectories are described). We have: $H^{\tau(x,y)}(x,y)=\Phi^{\delta(x,y)}(P(x,y))$ where $\delta_L(x,y):=\tau_L(x,y)-\tau^*_L(x,y)\in[0,1]$. Observe that $\tau_L(x,y)$ and $\tau^*_L(x,y)$ depend only on $x$ because $V(x,y)=V(x,z)$ for all $(x,y),(x,z)\in[-1/2,1/2]\times[-1/2,-1/6]$. Using flow boxes, it is easily seen that $\delta_L(x,y)=x+const$. The claimed uniformity follows from the boundedness of $\delta_L$: item \eqref{item:transit1} is proved.

Finally,  item \eqref{item:transit2} is a consequence of the following easy estimates.
 \begin{itemize}
  \item $x(0)\geq 4/10$;
  \item $x(t_1)=1/12=1/2-\Lambda^{t_1}1/10$ so $t_1=\log(25/6)/\lambda$;
  \item $x(t_2)=-1/12\geq -1/2+K^{-(t_2-t_1)}7/12$, so $t_2-t_1\geq \log(7/5)/\kappa$;
  \item $x(t_3)=-5/12= -1/2+K^{-(t_3-t_4)}5/12$, so $t_3-t_2=\log 5/\kappa$.
\end{itemize}
Indeed, these imply, writing $\bar y(t)$ for $y(t)+1/2$,
 \begin{itemize}
  \item $\bar y(t_1)=K^{-t_1}\bar y(0)=(25/6)^{-\eta}\bar y(0)$;
  \item $\bar y(t_2)=K^{-(t_2-t_1)}\bar y(t_1)\leq (5/7)\bar y(t_1)$;
  \item $\bar y(t_3)\leq \Lambda^{t_3-t_2}\bar y(t_2) \leq 5^{1/\eta}\bar y(t_2)$.
 \end{itemize}
Finally, $\bar y(\tau)\leq\Lambda \bar y(t_3)$ and
 $$
     \bar y(t_3)\leq  5^{1/\eta}(25/6)^{-\eta}(5/7)\bar y(t_0)\leq \exp(1.61/\eta-1.42\eta-0.33)\bar y(t_0),
 $$
(recall that $K$ and therefore $\eta$ is large). This implies item \eqref{item:transit2} of the Proposition.

\medbreak

It remains to see that $\lambda(H)=\log\Lambda$. We have arranged that $\Lip(H|Q\cup(D\setminus[-7/12,7/12]^2))<\Lambda$.  Clearly, $\lambda(H|\partial[-1/2,1/2]^2)=\log\Lambda$ as the time spent outside a neighborhood of the corners is uniformly bounded. This extends to $[-7/12,7/12]^2\setminus]-1/2,1/2[^2$. For the remaining part, observe that by item \eqref{item:filter}, there are constant $C,k_0$ such that:
 $$
   \Lip(H^k|[-1/2,1/2]^2\setminus Q)\leq C \Lip(H^{k-k_0}|[-1/2,1/2]^2\setminus[-5/12,5/12]^2).
 $$
As all points in this latter subset converge to $\partial [-1/2,1/2]^2$, the corresponding Lipschitz constant is bounded by $\Lambda^k$. Thus $\lambda(H)\leq\log\Lambda$.

\medbreak

Proposition \ref{prop:factsH} is proved.
\end{proof}

\bibliographystyle{plain}
\bibliography{buzzimath}

\end{document}